\documentclass[11pt]{amsart}

 \usepackage{amsmath,amsthm,amsfonts,amssymb,verbatim}
 \usepackage{url}
 \usepackage{graphicx}
 \usepackage[all]{xy}
 \usepackage{wrapfig}
 \usepackage{pinlabel}
 \usepackage{subfigure}

\def\co{\colon\thinspace}

\newcommand{\LO}{\operatorname{LO}}

\newcommand{\sT}{\mathcal{T}}

\newcommand{\bZ}{\mathbb{Z}}

\newcommand{\rk}{\operatorname{rk}}

\newtheorem{theorem}{Theorem}

\newtheorem{definition}[theorem]{Definition}

\newtheorem{corollary}[theorem]{Corollary}
\newtheorem{proposition}[theorem]{Proposition}

\newtheorem{lemma}[theorem]{Lemma}

\title[Graph manifolds, left-orderability and amalgamation]{Graph manifolds, left-orderability and amalgamation}
\date{June 2, 2011}

\author[Adam Clay]{Adam Clay}
\address{CIRGET,  Universit\'e du Qu\'ebec \`a Montr\'eal, Case postale 8888, Succursale centre-ville, Montr\'eal QC, H3C 3P8.}
\email{aclay@cirget.ca}

\author[Tye Lidman]{Tye Lidman}
\address{Department of Mathematics, UCLA, 520 Portola Plaza, Los Angeles CA, 90095.}
\email{tlid@math.ucla.edu}

\author[Liam Watson]{Liam Watson}
\address{Department of Mathematics, UCLA, 520 Portola Plaza, Los Angeles CA, 90095.}
\email{lwatson@math.ucla.edu}

\thanks{First and third authors partially supported by NSERC postdoctoral fellowships}

\begin{document}

\begin{abstract}
We show that every irreducible toroidal integer homology sphere graph manifold has a left-orderable fundamental group. This is established by way of a specialization of a result due to Bludov and Glass \cite{BG2009} for the almagamated products that arise, and in this setting work of Boyer, Rolfsen and Wiest \cite{BRW2005} may be applied. Our result then depends on input from 3-manifold topology and Heegaard Floer homology. 
\end{abstract}

\maketitle

\section{Introduction}

A group $G$ is said to be \textit{left-orderable} if there exists a strict total ordering $<$ of $G$ such that $g<h$ implies $fg<fh$ for all $f, g, h$ in $G$.  This is equivalent to the existence of a \textit{positive cone} $P \subset G$, which is a subset of elements of $G$ satisfying $P \cdot P \subset P$, and $P \sqcup \{ 1 \} \sqcup P^{-1} = G$.    To see that these notions are equivalent, if one is given a left ordering of a group $G$, observe that
\[ P := \{ g \in G | g>1 \}
\]
is a positive cone; conversely if $P \subset G$ is a positive cone then
\[ g <h  \Longleftrightarrow g^{-1}h \in P
\]
defines a left ordering of $G$. We adopt the convention that $P\ne\emptyset$, so that the trivial group is not left-orderable.

Left-orderability is not preserved under many classical group operations.  For example, if $G$ is a left-orderable group it is easy to see that a quotient of $G$ may or may not be left-orderable.  In fact it is a relatively simple exercise to prove that a quotient $G/N$ is left-orderable if and only if $N$ is \textit{relatively convex} in some left ordering of $G$.  On the other hand, the question of left-orderability of other classical constructions is not so simple.  Both HNN-extensions and free products with amalgamation present considerable difficulty when attempting to determine necessary and sufficient conditions for them to be left-orderable, and related problems remained open for many years (see for example \cite[Problem 15.34]{Kourovka02}).  However with recent work, left-orderability of free products with amalgamation and graphs of groups is finally well understood \cite{BG08, BG2009, Chiswell2010}.

At the same time, there has been increasing interest in left-orderable groups from a topological perspective, with particular interest in the left-orderability of the fundamental groups of 3-manifolds.  Many tools in 3-manifold theory involve decomposing the manifold into simpler pieces along incompressible surfaces, resulting in a fundamental group that can be presented as a free product with amalgamation or a graph of groups.   As such, these new results from the field of orderable groups should facilitate left-ordering of the fundamental groups of many 3-manifolds that contain incompressible surfaces.   In particular, if a 3-manifold $M$ admits a JSJ decomposition (see Section \ref{sec:topology}, in particular Theorem \ref{thm:JSJ}) into components that are well-understood from an orderability standpoint, then it should be possible to determine whether or not $\pi_1(M)$ is left-orderable.  Our main theorem is the product of this approach.

\begin{theorem}\label{thm:main}  Let $Y$ be an irreducible, toroidal graph manifold.  If $Y$ is an integer homology 3-sphere, then $\pi_1(Y)$ is left-orderable.
\end{theorem}

\subsection*{A remark on conventions} We include Seifert fibred spaces as graph manifolds; these arise when the JSJ decomposition is trivial. As such, not every toroidal graph manifold admits a non-trivial JSJ decomposition (though the converse always holds).  We also include the connect sum of graph manifolds as a graph manifold. Indeed, Theorem \ref{thm:main} is a special case of Theorem \ref{thm:general-main} (proved in Section \ref{sec:topology}) treating the reducible case.

\subsection*{Organization} The plan of this paper is as follows.  In Section \ref{sec:grouptheory}, we review the necessary definitions and recent results pertaining to left-orderability of amalgamated free products.  We also prove our main topological application of these results, which connects the JSJ decomposition of a 3-manifold with the manifolds obtained by Dehn filling the JSJ-pieces along their torus boundaries.  Section \ref{sec:topology} is devoted to  the proof  of Theorem \ref{thm:main}, and covers the necessary elements from Heegaard Floer homology. As a quick consequence, we answer a question of D\c{a}bkowski, Przytycki and  Togha \cite[Problem 1(iii)]{DPT2005}. This, together with other examples and generalizations, is provided in Section \ref{sec:examples}.

\subsection*{Acknowledgements} We thank Steve Boyer for his interest in this project, and in particular for posing the question `does the result of $\pm 4$-surgery on the figure eight have a left-orderable fundamental group?' (see \cite{BGW2010}) to which the answer is `yes' (see Section \ref{sec:examples}). We also thank Stefan Friedl for comments on an early draft of this paper. 

\section{Requisite group theory and relevant background}

\label{sec:grouptheory}

We begin with the definitions that are required to state a result of Bludov and Glass \cite{BG2009}.   If a group $G$ is left-orderable, we denote the set of all positive cones in $G$ by $\LO(G)$, we think of $\LO(G)$ as the set of all left-orderings of $G$.  The set $\LO(G)$ comes equipped with a natural $G$-action by conjugation, defined by $g(P) = gPg^{-1}$ for all $g \in G$ and $P \in \LO(G)$.  In terms of left-orderings, this action can be described as follows: given a left-ordering $<$ of $G$, an element $g$ sends $<$ to the ordering $<^g$, which is defined according to the rule
\[ h<^g f  \Longleftrightarrow hg <fg.
\]

\begin{definition} A set $L \subset \LO(G)$ of left-orderings is \textit{normal} if it is $G$-invariant.
\end{definition}

\begin{definition}  Suppose that $G_i$ is a left-orderable group, $H_i$ is a subgroup of $G_i$ and let $L_i \subset \LO(G_i)$ denote a set of left-orderings of $G_i$ for $i=1,2$.  Suppose that $\phi:H_1 \rightarrow H_2$ is an isomorphism.  Then $\phi$ is \textit{compatible for $(L_1, L_2)$} if
\[ (\forall P_1 \in L_1)(\exists P_2 \in L_2) \mbox{ such that } (\forall h_1 \in H_1)( h_1 \in P_1 \Rightarrow \phi(h_1) \in P_2).
\]
\end{definition}

One of the main results of \cite{BG2009} is the following.

\begin{theorem}[Bludov-Glass {\cite[Theorem A]{BG2009}}]
\label{thm:amalgam}
Suppose that $G_i$ are left-orderable groups, and $H_i$ is a subgroup of $G_i$ for $i=1,2$.  Let $\phi: H_1 \rightarrow H_2$ be an isomorphism.  The free product with amalgamation $G_1 *G_2 (H_1 \stackrel{\phi}{\cong} H_2)$ is left-orderable if and only if there exist normal families $L_i \in \LO(G_i)$ $(i=1,2)$ such that $\phi$ is compatible for $(L_1, L_2)$.
\end{theorem}

In general, using this theorem to left-order free products with amalgamation seems to be quite difficult.  Unless the groups in question have a very well-understood structure, verifying the existence of normal, compatible families of left-orderings can be an intractable problem (It is possible in some cases, see  Section \ref{sec:examples}).

However, there are several immediate corollaries of this theorem that are very useful when attempting to left-order amalgamated products, as the compatibility and normality conditions become easier to verify.  The following will be most useful for our purposes.

\begin{corollary} [Bludov-Glass {\cite[Corollary 5.3]{BG2009}}]
\label{thm:cyclicprod}
 Suppose that $G_i$ are left-orderable groups with rank one abelian subgroups $H_i$, $i=1,2$.  Let $\phi: H_1 \rightarrow H_2$ be an isomorphism.   Then $G_1 *G_2 (H_1 \stackrel{\phi}{\cong} H_2)$ is left-orderable.
\end{corollary}

Our main group-theoretic result will be an application of Theorem \ref{thm:amalgam} to the decomposition of a 3-manifold $M$ along an incompressible torus into components $M_1$ and $M_2$ having toroidal boundaries.  The key observation is that in this special case, applying Theorem \ref{thm:amalgam} can be reduced to an application of Corollary \ref{thm:cyclicprod} by using Dehn filling and the following theorem from 3-manifold topology:

\begin{theorem} [Boyer-Rolfsen-Wiest {\cite[Theorem 1.1]{BRW2005}}]
\label{thm:loimage}
 Suppose that $M$ is a compact, connected, $\mathbb{P}^2$-irreducible 3-manifold.  Then $\pi_1(M)$ is left-orderable if and only if there exists a nontrivial homomorphism from $\pi_1(M)$ to a left-orderable group.
\end{theorem}

To state and prove our main group-theoretic result, we recall the notion of Dehn filling.  Suppose that $M$ is a 3-manifold with incompressible torus boundary.  A slope $\alpha$ is a primitive element in the projective homology $H_1(\partial M;\bZ) / \{ \pm 1\}$ of the boundary.   The result of Dehn filling $M$ along the slope $\alpha$ is the 3-manifold $M(\alpha)$ obtained by identifying the boundary of a solid torus $D^2 \times S^1$ to $\partial M$ in such a way that $\partial D^2 \times \{*\}$ is glued to $\alpha$. The following definition is natural in this setting:

\begin{definition}\label{def:LOslope}
Given a 3-manifold $M$ with torus boundary, a slope $\alpha$ will be called left-orderable if $\pi_1(M(\alpha))\cong \pi_1(M) / \langle \langle \alpha \rangle \rangle$ is a left-orderable group.
\end{definition}

It is our convention that the trivial group is not left-orderable, so any left-orderable slope $\alpha$ generates a proper normal subgroup $\langle\langle\alpha\rangle\rangle$ of $\pi_1(M)$ and $\pi_1(M(\alpha))$ is non-trivial.

\begin{theorem}
\label{thm:JSJLO}
Suppose that $M_1$ and $M_2$ are $3$-manifolds with incompressible torus boundaries, and $\phi : \partial M_1 \rightarrow \partial M_2$ is a homeomorphism such that $Y = M_1 \cup_{\phi} M_2$ is irreducible. If there exists a left-orderable slope $\alpha$ such that $\phi_*(\alpha)$ is also a left-orderable slope, then $\pi_1 (Y)$ is left-orderable (here, $\phi_*$ is the induced homomorphism on fundamental groups).
\end{theorem}
\begin{proof}
Let $G_i$ denote the fundamental group $\pi_1(M_i)$ for $i=1,2$, each equipped with an inclusion $f_i:\mathbb{Z} \oplus \mathbb{Z} \rightarrow G_i$ that identifies the peripheral subgroup with $\mathbb{Z} \oplus \mathbb{Z}$, satisfying $\phi_* \circ f_1 = f_2$.   Write  $q_1 : G_1 \rightarrow G_1/ \langle \langle \alpha \rangle \rangle$ and $q_2 : G_2 \rightarrow G_2 / \langle \langle \phi_*(\alpha) \rangle \rangle$ for the natural quotient maps.

Suppose that $\alpha$  and $\phi_*(\alpha)$ are left-orderable slopes, and consider $\langle \langle \alpha \rangle \rangle \cap \pi_1(\partial M_1)$.   Since this intersection is a nontrivial subgroup of $\pi_1(\partial M_1) \cong \mathbb{Z} \oplus \mathbb{Z}$ and $\alpha$ is primitive, the intersection is isomorphic to either $\mathbb{Z} \cong \langle \alpha \rangle, \mathbb{Z} \oplus n \mathbb{Z} \subset \pi_1(\partial M_1)$, or $\mathbb{Z} \oplus \mathbb{Z} \cong \pi_1(\partial M_1)$.  If $\langle \langle \alpha \rangle \rangle \cap \pi_1(\partial M_1) \cong \mathbb{Z} \oplus n \mathbb{Z}$, then the quotient $\pi_1(M(\alpha))$ would have torsion, so this case does not arise when $\alpha$ is a left-orderable slope.  The same observation holds for the left-orderable slope $\phi_*(\alpha)$, so  we break our proof into two cases.

First consider the case where $\langle \langle \alpha \rangle \rangle \cap \pi_1(\partial M_1)$ and $\langle \langle \phi_*(\alpha) \rangle \rangle \cap \pi_1(\partial M_2)$ are both infinite cyclic.
 Here,  $\phi$ induces an isomorphism $\bar{\phi}$ between subgroups
\[\bar{\phi}:  q_1(\pi_1(\partial M_1)) \rightarrow q_2(\pi_1(\partial M_2)),\]
satisfying $\bar{\phi} \circ q_1 \circ f_1 = q_2 \circ f_2$.  The subgroups  $q_1(\pi_1(\partial M_1))$ and $q_2(\pi_1(\partial M_2 ))$ are both infinite cyclic, and by the universal property for pushouts, we have a unique homomorphism
\[ h\co G_1*_{\phi}G_2 \longrightarrow G_1/\langle\langle \alpha \rangle\rangle*_{\bar{\phi}}G_2/\langle\langle \phi_*(\alpha) \rangle\rangle
\]
resulting from the following diagram:
\[\xymatrix@R=15pt@C=25pt{
{\bZ\oplus\bZ}\ar[r]^{f_2}\ar[d]_{f_1}& {G_2}\ar[d]\ar@/^1pc/[dr] ^{q_2}& \\
{G_1}\ar[r]\ar@/_1pc/[dr]_{q_1}  & {G_1 *_\phi G_2}\ar@{-->}[dr]^h & {G_2/\langle\langle\phi_*(\alpha)\rangle\rangle}\ar[d] \\
& {G_1/\langle\langle\alpha\rangle\rangle}\ar[r] & {G_1/\langle\langle\alpha\rangle\rangle *_{\bar{\phi}}G_2/\langle\langle\phi_*(\alpha)\rangle\rangle}
}.\]
Note that $h$ is nontrivial (in fact, surjective) since the maps $q_1$ and $q_2$ are surjective.

Because the slopes $\alpha$ and $\phi_*(\alpha)$ are both left-orderable slopes, the group $ G_1/\langle\langle \alpha \rangle\rangle*_{\overline{\phi}}G_2/\langle\langle \phi_*(\alpha) \rangle\rangle$ is a free product of left-orderable groups amalgamated along a cyclic subgroup.   The image of the map $h$ is therefore a left-orderable group by Theorem \ref{thm:cyclicprod}, so that $\pi_1(Y) \cong G_1 *_\phi G_2$ is left-orderable, by Theorem \ref{thm:loimage}.

On the other hand, suppose that either $\langle \langle \alpha \rangle \rangle \cap \pi_1(\partial M_1) = \pi_1(\partial M_1)$, or $\langle \langle \phi_*(\alpha) \rangle \rangle \cap \pi_1(\partial M_2) = \pi_1(\partial M_2)$, or both.   Without loss of generality, suppose that $\alpha$ satisfies $\langle\langle\alpha\rangle\rangle \cap \pi_1(\partial M_1)= \pi_1(\partial M_1)$. In this setting we have an alternative construction for  $h$ as follows.
\[\xymatrix@R=15pt@C=25pt{
{\bZ\oplus\bZ}\ar[r]^{f_2}\ar[d]_{f_1}& {G_2}\ar[d]\ar@/^1pc/[ddr] ^{1}& \\
{G_1}\ar[r]\ar@/_1pc/[drr]_{q_1}  & {G_1 *_\phi G_2}\ar@{-->}[dr]^h &  \\
& & {G_1/\langle\langle\alpha\rangle\rangle}
}\]
Note that this is well defined since $\langle\langle\alpha\rangle\rangle$ contains the entire peripheral subgroup $\pi_1(\partial M_1)$, and $h$ is again surjective. Now as $\alpha$ is a left-orderable slope, $G_1/\langle\langle\alpha\rangle\rangle$ is left-orderable and $h$ provides the required homomorphism to a left-orderable group so that $\pi_1(Y)$ is left-orderable by Theorem \ref{thm:loimage}.
\end{proof}
We record an immediate consequence. Recall that a splicing of knots $K_1$ and $K_2$ in homology spheres $Y_1$ and $Y_2$ is a 3-manifold $Y=M_1\cup_\phi M_2$ where $\phi(\mu_1)=\lambda_2$ and $\phi(\lambda_1)=\mu_2$. Here, the pair $\{\mu_i,\lambda_i\}$ is the preferred framing for each knot $K_i$.  In particular, $\lambda_i$ bounds a surface in $M_i= Y_i\smallsetminus K_i$ for $i=1,2$. It follows by construction that $Y$ is  an integer homology sphere.
\begin{corollary}\label{cor:splice}
Let $Y$ be the result of splicing knots $K_1$ and $K_2$ in irreducible integer homology spheres $Y_1$ and $Y_2$ respectively. If $\pi_1(Y_1)$ is left-orderable and $M_2(\lambda_2)$ is prime, then $\pi_1(Y)$ is left-orderable.
\end{corollary}
\begin{proof}Since $\pi_1(Y_1)$ is left-orderable, $\mu_1$ is a left-orderable slope. On the other hand, $\lambda_2$ is a left-orderable slope since $M_2(\lambda_2)$ has positive first Betti number, so there exists a surjection $\pi_1(M_2(\lambda_2)) \rightarrow \mathbb{Z}$. Here we apply Theorem  \ref{thm:loimage} making use of the assumption that $M_2(\lambda_2)$ is prime. Now since $\phi(\mu_1)=\lambda_2$, the result follows from Theorem \ref{thm:JSJLO}.\end{proof} 

\section{The proof of Theorem \ref{thm:main}}
\label{sec:topology}
\subsection{JSJ decompositions for integer homology spheres}

Given a 3-manifold $Y$ containing an essential torus $T$ we have that $\pi_1(T)\cong\bZ\oplus\bZ$ is a subgroup of $\pi_1(Y)$. As a result, $\pi_1(Y)$ may be viewed naturally as a free product with amalgamated subgroup  $\bZ\oplus\bZ$. This observation is particularly important in light of the following result of Jaco and Shalen \cite{JS1976} and Johannson \cite{Johannson1975}.

\begin{theorem}[Jaco-Shalen, Johannson]\label{thm:JSJ}Every compact, orientable, irreducible 3-manifold $Y$ with torus boundary components contains a (possibly empty) minimal collection of essential tori $\sT$, unique up to isotopy, such that each component that results from cutting $Y$ (denoted $Y\smallsetminus\sT$) along $\sT$ is either Seifert fibred or atoroidal. \end{theorem}

The resulting decomposition of $Y$ along such a collection of tori is referred to as the JSJ decomposition of $Y$. Notice that there is a natural quotient from $Y$ to an underlying graph $\Gamma_Y$ that encodes the JSJ decomposition: the vertices are obtained by collapsing the components of $Y\smallsetminus\sT$ to points, and the edges correspond to collapsing $T\times I$ to the interval $I$, for each torus $T\in\sT$. As a result we have a surjection $H_1(Y;\bZ)\to H_1(\Gamma_Y;\bZ)$ so that whenever $H_1(Y;\bZ)=0$ the corresponding graph $\Gamma_Y$ is a tree. In particular, given a torus $T\in \sT$ we have that $Y\smallsetminus T$ is disconnected if $Y$ is an integer homology sphere; this observation will provide a key step to our induction in the proof of Theorem \ref{thm:main}.

\begin{lemma}\label{lem:toroidal-ZHS}Let $Y$ be a toroidal integer homology sphere, and suppose that $Y=  M_1 \cup_\phi M_2$ where $\phi:\partial M_1 \to \partial M_2$ is a homeomorphism identifying the boundaries.   Then $Y$ may be viewed as a splicing of knots $K_i$ in integer homology spheres $Y_i$ for $i=1,2$. \end{lemma}
\begin{proof} Let $Y$ be a toroidal integer homology sphere, and denote the connected components of $Y\smallsetminus T$ by $M_1$ and $M_2$. Since $Y$ is an integer homology sphere, an application of the Mayer-Vietoris theorem gives an isomorphism \[H_1(S^1\times S^1;\bZ)\cong H_1(M_1;\bZ)\oplus H_1(M_2;\bZ).\] Applying the long exact sequence for the pair $(M_i,\partial M_i)$ we conclude that $H_1(M_i;\bZ)\cong\bZ$ for $i=1,2$; let $\lambda_i$ denote the longitudinal slope in $\partial M_i$ (the unique  simple closed curve that bounds in $M_i$).  Note that $|H_1(M_1\cup_f M_2;\bZ)| = \Delta(f(\lambda_1),\lambda_2)$ for any homeomorphism $f\co \partial M_1\to \partial M_2$, where $\Delta$ measures the minimal geometric intersection number between slopes.

Now denote by  $\phi\co \partial M_1\to \partial M_2$ the homeomorphism  reconstructing $Y\cong M_1\cup_\phi M_2$. This homeomorphism yields a preferred meridian for each $M_i$ as follows. Since $\Delta(\phi(\lambda_1),\lambda_2)=1$ we define $\mu_1 = \phi^{-1}(\lambda_2)$ and and $\mu_2=\phi(\lambda_1)$. By construction, the pair $\{\mu_i,\lambda_i\}$ yields a basis for the peripheral subgroup of each $\pi_1(M_i)$. Now specifying $K_i$ in each $Y_i=M_i(\mu_i)$ as the core of the surgery torus produces a knot with meridian $\mu_i$ and longitude $\lambda_i$. This  shows that $Y$ is a splicing of knots $K_1$ and $K_2$ in the sense that $\phi(\mu_1)=\lambda_2$ and $\phi(\lambda_1)=\mu_2$.\end{proof}

We will always make this choice of preferred basis where possible. In this way, $n$-surgery and $\frac{1}{n}$-surgery  on $K_i$ are well defined as $n\mu+\lambda$ and $\mu+n\lambda$ Dehn filling on $M_i$, respectively. Moreover, such a choice lets us unambiguously refer to a slope $p\mu+q\lambda$ as a reduced rational $\frac{p}{q}$, where $\frac{1}{0}$ denotes the trivial surgery, so that $M(\frac{p}{q})$ denotes Dehn filling along the slope $\frac{p}{q}$.  Note that it is not restrictive to assume that $q\ge0$. By construction, $H_1(M(\frac{p}{q})) \cong \mathbb{Z}/|p|\mathbb{Z}$.

With these observations and conventions in place, the family of manifolds considered in Theorem \ref{thm:main} are as follows:
\begin{definition}\label{def:graph} A graph manifold is a compact, orientable, connected 3-manifold for which each prime component $Y$ has JSJ decomposition along a family of tori $\sT$ such that the components of $Y\smallsetminus \sT$ are Seifert fibred. \end{definition}

This class of manifolds was introduced by Waldhausen \cite{Waldhausen1967}. Note that Seifert fibred spaces are graph manifolds; these arise precisely when $\sT$ is empty. This definition also allows for graph manifolds obtained by connect sum of graph manifolds. 

Our proof of Theorem \ref{thm:main} will be an induction on the number of tori in the collection $\sT$, so we record the following fact.
\begin{lemma}
\label{lem:inductionstep}
Let $Y$ be a graph manifold with torus boundary, and minimal collection $\sT$ of tori decomposing $Y$ into Seifert fibred components $M_1, \cdots, M_k$.  Let $\alpha$ be any slope on the boundary of $Y$.  Then $Y(\alpha)$ is also a graph manifold, and $Y(\alpha)$ admits a minimal collection of tori $\sT'$ decomposing $Y(\alpha)$ into Seifert fibred pieces with $| \sT' | \leq |\sT|$.
\end{lemma}
\begin{proof} Observe that each torus in the collection $\sT$ embeds naturally in the manifold $Y(\alpha)$.
If $M_i$ denotes the Seifert fibred component of $Y \smallsetminus \sT$ that contains $\partial Y$, then $Y(\alpha) \smallsetminus \sT$ has components $M_1, \cdots, M_i(\alpha), \cdots, M_k$.   By \cite{Heil1974}, the manifold $M_i(\alpha)$ is either Seifert fibred, or is a connect sum of lens spaces and possibly copies of $S^1 \times S^2$.  In either case, since every manifold $M_1, \cdots, M_i(\alpha), \cdots, M_k$ is either Seifert fibred or a connect sum of Seifert fibred pieces, the collection $| \sT|$ of tori embedded in $Y(\alpha)$ will cut the prime components of $Y(\alpha)$ into Seifert fibred pieces.  Therefore $Y(\alpha)$ is a graph manifold.  Moreover,  if $| \sT'|$ is the collection of cutting tori from the JSJ decomposition of $Y(\alpha)$, then $| \sT' | \leq |\sT|$ since $\sT'$ is minimal.
\end{proof}

A similar argument yields the following.

\begin{lemma}
\label{lem:irreducible}
Let $Y$ be an irreducible integer homology sphere graph manifold.  Decompose $Y$ into $M_1 \cup_\phi M_2$ such that $M_2$ is Seifert fibred.  Then $Y_1 = M_1(\mu_1)$ is irreducible, where $M_1$ is $Y_1 \smallsetminus K_1$ as in Lemma~\ref{lem:toroidal-ZHS}.
\end{lemma}
\begin{proof}
By work of Heil \cite{Heil1974}, the result holds when $Y_1$ is Seifert fibred, since an integer homology sphere cannot contain a lens space or $S^2 \times S^1$ connect summand.  More generally, $Y_1$ is a graph manifold with some minimal set of tori $\sT_1$ and we proceed as in the proof of Lemma~\ref{lem:inductionstep}.  
\end{proof}

For repeated use in this paper, we single out a particular fact used in these arguments: Heil's work shows that both the $0$-filling and every $\frac{1}{n}$-filling on a Seifert fibred $M$ with $H_1(M;\mathbb{Z}) \cong \mathbb{Z}$ is both Seifert fibred and prime; the same also holds for all but finitely many $\frac{n}{1}$-fillings \cite{Heil1974,Scott1973}.

\subsection{Input from Heegaard Floer homology}

In this section we gather material from Heegaard Floer homology and collect some immediate corollaries that will be used to prove Theorem \ref{thm:main}.  Recall that an L-space is a rational homology sphere with the simplest possible Heegaard Floer homology in the sense that \[|H_1(Y;\mathbb{Z})| = \rk \widehat{HF}(Y).\] A large class of L-spaces, relevant in the present context, is provided by the following:

\begin{theorem}[Boyer-Gordon-Watson \cite{BGW2010}]\label{thm:BGWlspace} Suppose that $Y$ is a closed, connected, orientable Seifert fibred 3-manifold. Then $Y$ is an L-space if and only if $\pi_1(Y)$ is not left-orderable.
\end{theorem}

We remark that the case wherein the base orbifold of the Seifert fibred space is restricted to $S^2$ was observed independently by Peters \cite{Peters2009}; the proof of the general case also appears in \cite{Watson2009}.  It is possible to specify exactly which Seifert fibred integer homology spheres have fundamental groups which are not left-orderable.

\begin{theorem}[Boyer-Rolfsen-Wiest {\cite[Corollary 3.12]{BRW2005}}]\label{thm:BRWlospheres} Suppose that $Y$ is a Seifert fibred integer homology sphere.  If $\pi_1(Y)$ is not left-orderable, then $Y$ is either $S^3$ or the Poincar\'e homology sphere $\Sigma(2,3,5)$.\end{theorem}

The analogous result (and immediate corollary of Theorem \ref{thm:BGWlspace} and Theorem \ref{thm:BRWlospheres}) from Heegaard Floer homology was first proved directly by Eftekhary \cite[Theorem 1.1]{Eftekhary2009}. Namely, $S^3$ and $\Sigma(2,3,5)$ are the only Seifert fibred integer homology sphere L-spaces. 

This points to a strategy of proof for Theorem \ref{thm:main} that employs Theorem \ref{thm:JSJLO}. Namely, we may induct in the number of Seifert pieces of an graph manifold provided some control of $S^3$ and $\Sigma(2,3,5)$ (in particular, Seifert fibred knots in these manifolds) may be established. To this end some additional tool must be used, and in this particular setting it seems most natural to make use of stringent restrictions imposed by Heegaard Floer homology. With this strategy in mind, our next result will make use of work of Wu \cite{ZWu2009}.

\begin{theorem}[Wu {\cite[Theorem 1.3]{ZWu2009}}]\label{thm:Wu-cosmetic}  Suppose that $Y$ is an integer homology sphere L-space containing the knot $K$, and set $M = Y \smallsetminus K$.  If $K$ is non-trivial, then no two slopes $r$ and $r'$ of the same sign yield orientation-preserving homeomorphic manifolds $M(r)$ and $M(r')$.
\end{theorem}

\begin{lemma}\label{lem:smallsurgeryLspace} Let $K$ be a non-trivial knot in $Y=S^3$ or $\Sigma(2,3,5)$ with Seifert fibred exterior.  Then there are at most finitely many slopes of the form $\frac{1}{n}$ (with $n \in \mathbb{Z}$) such that $\frac{1}{n}$-surgery on $K$ yields an L-space.
\end{lemma}
\begin{proof}
Setting $M=Y\smallsetminus K$, each $M(\frac{1}{n})$ is a Seifert fibred space.  Any L-space that results from $\frac{1}{n}$-surgery on $K$ is therefore homeomorphic to either $S^3$ or $\Sigma(2,3,5)$ \cite[Theorem 1.1]{Eftekhary2009}.  Note that our lemma is now a question of cosmetic surgeries: how many surgeries on $K$ yield homeomorphic manifolds?

Suppose there were infinitely many slopes.  By the pigeonhole principle, there would have to be at least two slopes of the same sign that yield orientation-preserving homeomorphic manifolds. Since $Y$ is an integer homology sphere L-space by assumption, we may apply Theorem \ref{thm:Wu-cosmetic} to arrive at a contradiction.
\end{proof}

This yields an immediate corollary.
\begin{corollary}\label{cor:smallsurgeryLO} Let $K$ be a non-trivial knot in $Y=S^3$ or $\Sigma(2,3,5)$ such that the complement $Y \smallsetminus K$ is Seifert fibred.  At most finitely many of the slopes $\frac{1}{n}$ are not left-orderable.
\end{corollary}
\begin{proof}
By Lemma \ref{lem:smallsurgeryLspace}, at most finitely many of the manifolds $M(\frac{1}{n})$ are L-spaces.  By Theorem \ref{thm:BGWlspace}, since each $M(\frac{1}{n})$ is Seifert fibred, we conclude that at most finitely many of the slopes $\frac{1}{n}$ are not left-orderable.\end{proof}

\begin{lemma}\label{lem:largesurgeryLspace} Let $K$ be a non-trivial knot in $Y=S^3$ or $\Sigma(2,3,5)$  with Seifert fibred exterior.  For either all positive or all negative integers $n$, the result of $n$-surgery on $K$ is not an L-space.
\end{lemma}
\begin{proof}
Arguing by contradiction, assume there exist both positive and negative integral L-space fillings of $K$. Let $M=Y\smallsetminus K$. 

If $Y$ is an L-space integer homology sphere, we may apply \cite[Proposition 3.5]{ZWu2009}, a generalization of \cite[Proposition 9.5]{OSz-rational}, which gives a formula for the total rank of the Heegaard Floer homology of $\frac{p}{q}$-surgery on $K$.  Without loss of generality, we assume $q>0$.  
The rational surgery formula reads
\[
\rk \widehat{HF}(M(\textstyle{\frac{p}{q}})) = p + 2\max\{0,(2\nu_Y(K)-1)q-p\} + q\sum_{s \in \mathbb{Z}} (\rk H_*(A_s) - 1).
\]
The formula depends on the invariant $\nu_Y(K)$ being nonnegative.  While we have not defined this invariant, we will only need the following property: after possibly reversing the orientation of $Y$ and maintaining the original orientation on $K$, we can make $\nu_Y(K)$ to be nonnegative.  Therefore, we assume that $\nu_Y(K) \geq 0$ for the remainder of the proof.  Note that the pair $(Y,K)$ still admits positive and negative integral L-space fillings, as being an L-space is independent of orientation \cite[Proposition 2.5]{OSz2004-properties}.

We also do not define the chain-complexes $A_s$, but point out that $H_*(A_s)$ is isomorphic to the Heegaard Floer homology of large positive surgeries on $K$ in certain Spin$^c$ structures. In particular, $\rk H_*(A_s) \geq 1$ for all $s$.  By assumption, since there are positive, integral L-space fillings on $K$, we must have that $H_*(A_s)$ has rank 1 for all $s$. 

This formula leads to two natural cases.  First, suppose that $\nu_Y(K) = 0$, so that the surgery formula is given by
\[
\rk \widehat{HF}(M(\textstyle{\frac{p}{q}})) = |p| + q\sum_{s \in \mathbb{Z}} (\rk H_*(A_s) - 1).
\]
By this formula, $M(\frac{p}{q})$ must in fact be an L-space for any non-zero $\frac{p}{q}$.  However, this is a contradiction to Lemma~\ref{lem:smallsurgeryLspace}, which tells us that there exist $\frac{1}{n}$-fillings on $K$ which do not give L-spaces. As such, we conclude that $\nu_Y(K) > 0$.

In this case, the formula simplifies to 
\[\rk \widehat{HF}(M(\textstyle{\frac{p}{q}})) = p + 2\max\{0,(2\nu_Y(K)-1)q-p\}. \]
This shows that for $p < 0$,
\[
\rk \widehat{HF}(M(\textstyle{\frac{p}{1}})) = p + 2(2\nu_Y(K)-1-p) > |p|.
\]
Therefore, no negative integral surgeries can be L-spaces and we have found our contradiction.
\end{proof}

Again, we have the corresponding statement for left-orderability.

\begin{corollary}\label{cor:largesurgeryLO} Let $K$ be a nontrivial knot in $Y=S^3$ or $\Sigma(2,3,5)$ such that the complement is Seifert fibred.  Then for either infinitely many positive integers or infinitely many negative integers $n$, $\frac{n}{1}$ is a left-orderable slope.
\end{corollary}
\begin{proof}
The result of $\pm n$-surgery on $K$ is generically Seifert fibred.  Now we combine Lemma \ref{lem:largesurgeryLspace} and Theorem~\ref{thm:BGWlspace}.
\end{proof}

\subsection{A generalization and proof of Theorem \ref{thm:main}}

We now have the material in place to prove Theorem \ref{thm:main}, which is an immediate consequence of the following:

\begin{theorem}\label{thm:general-main}Let $Y$ be a graph manifold other than $S^3$, none of whose prime components is a Poincar\'e homology sphere.  If Y is an integer homology sphere, then $\pi_1(Y)$ is left-orderable.
\end{theorem}
\begin{proof}
First, observe that if $Y = Y_1 \# \cdots \# Y_m$, then $\pi_1(Y)$ is a free product of the groups $\pi_1(Y_i)$.  It is well known that if $G$ is a free product of the groups $G_i$, then $G$ is left-orderable if and only if each nontrivial $G_i$ is left-orderable \cite{Vino49}.  Therefore, $\pi_1(Y)$ will be left-orderable if and only if each nontrivial $\pi_1(Y_i)$ is left-orderable.  Hence, for the remainder of the proof we may assume that $Y$ is prime.

We now proceed by induction on the number of tori in the family $\sT$ provided by the JSJ decomposition of $Y$, as given by Theorem \ref{thm:JSJ}.  First, if $\sT = \emptyset$, then $Y$ is Seifert fibred.  Since $Y$ is not $S^3$ or $\Sigma(2,3,5)$ by assumption, Theorem~\ref{thm:BRWlospheres} ensures that $\pi_1(Y)$ is left-orderable.

The first case of interest is when $|\sT | = 1$.  In particular, $Y$ is a splicing of $Y_1$ and $Y_2$ along knots $K_1$ and $K_2$, as in Lemma \ref{lem:toroidal-ZHS}, such that each knot complement $M_i = Y_i \smallsetminus K_i$ is a boundary-irreducible Seifert fibred space.  We denote the gluing map  by $\phi$, so that $Y = M_1 \cup_{\phi} M_2$.  Observe that $Y_1$ and $Y_2$ must be Seifert fibred spaces as well, since they are integer homology spheres arising from Dehn filling the Seifert fibred manifolds $M_1$ and $M_2$ respectively.

Suppose first that one of the $Y_i$ has a left-orderable fundamental group. Without loss of generality, $\pi_1(Y_1)$ is left-orderable and so the meridian of $K_1$ is a left-orderable slope on $\partial M_1$.  Furthermore, $M_2(\lambda_2)$ is prime since $M_2$ is Seifert fibred.  By Corollary~\ref{cor:splice}, $\pi_1(Y)$ must be left-orderable as well.

Therefore, we can assume $\pi_1(Y_1)$ and $\pi_1(Y_2)$ are both non-left-orderable groups, so that each $Y_i$ is one of $S^3$ or $\Sigma(2,3,5)$.    By Corollary~\ref{cor:largesurgeryLO}, either infinitely many positive or negative integer slopes on $\partial M_1$ are left-orderable;  by Corollary \ref{cor:smallsurgeryLO} only finitely many slopes on $\partial M_2$ of the form $\frac{1}{n}$ are not left-orderable.   Recall that since the homeomorphism $\phi$ specifies a splicing of knots, relative to our chosen bases of the periphal subgroups $\phi$ sends the slope $\frac{p}{q}$ on $\partial M_1$ to the slope $\frac{q}{p}$ on $\partial M_2$.  Therefore, we can find a pair of left-orderable slopes of the form $\frac{n}{1}$ on $M_1$ and $\frac{1}{n}$ on $M_2$ that are identified by $\phi$.  Theorem~\ref{thm:JSJLO} completes the proof in this setting, and provides a base case for induction.

We now provide the induction step.  Assume that every irreducible integer homology sphere graph manifold whose JSJ decomposition has fewer than $p$ tori (other than $S^3$ or $\Sigma(2,3,5)$) has left-orderable fundamental group.  Suppose that $Y$ has tori $T_1,\ldots,T_p$ in the JSJ decomposition, where $p \geq 2$.  Since $\Gamma_Y$ is a tree, $T_p$ is a separating torus and $Y \smallsetminus T_p$ consists of the two manifolds $M_i = Y_i \smallsetminus K_i$, each containing $p-1$ or fewer tori in its JSJ decomposition.  Without loss of generality, we can choose $T_p$ such that $M_2$ is Seifert fibred.  By Lemma \ref{lem:inductionstep}, the number of tori in the JSJ decomposition of $Y_1 = M_1(\mu_1)$ is less than or equal to $p-1$. $Y_1$ must be irreducible by Lemma~\ref{lem:irreducible}.

Again, there are two cases to consider. First suppose that $Y_1$ (the manifold which may contain incompressible tori) has $\pi_1(Y_1)$ left-orderable.  Since $M_2(\lambda_2)$ is prime, we may apply Corollary~\ref{cor:splice} to see that $\pi_1(Y)$ is left-orderable. On the other hand,  if $\pi_1(Y_1)$ is not left-orderable, by induction, $Y_1$ must be $S^3$ or $\Sigma(2,3,5)$.  Since $Y_2$ is Seifert fibred, notice that we are once again considering a splicing of knots in Seifert fibred spaces. As above, the case wherein $\pi_1(Y_2)$ is left-orderable is handled by Corollary \ref{cor:splice}; if $\pi_1(Y_2)$ is not left-orderable then there exist left-orderable slopes $\frac{1}{n}$ and $\frac{n}{1}$ in $\partial M_1$ and $\partial M_2$ respectively and an application of Theorem \ref{thm:JSJLO} concludes the proof. We remark that this is not, strictly speaking, a second application of the case $|\sT|=1$: since the manifolds $M_1(\frac{1}{n})$ are generically integer homology sphere graph manifolds with non-trivial JSJ decomposition, this final step makes use of the induction hypothesis. \end{proof} 


\section{Examples, consequences and generalizations}\label{sec:examples}

\subsection{On a question of D\c{a}bkowski, Przytycki and  Togha} As an immediate consequence of Theorem \ref{thm:main}, we provide a negative answer to \cite[Problem 1(iii)]{DPT2005}.

\begin{proposition}
There exists a hyperbolic knot $K$ in $S^3$ with $n$-fold cyclic-branched cover $\Sigma_n(K)$ that has left-orderable fundamental group and finite $H_1(\Sigma_n(K);\bZ)$ (for infinitely many positive integers $n$).
\end{proposition}
\begin{proof}
Let $K$ be the Conway knot (the knot $11n34$ in the standard knot tables; its mutant, the Kinoshita-Terasaka knot, is the knot $11n42$).  As $\det(K) = 1$, the two-fold cyclic branched cover of $K$ has trivial first-homology. Montesinos shows that $\Sigma_2(K)$ can be obtained by surgery on a connect-sum of two trefoil knots \cite{Montesinos75}.  This surgery results in an irreducible toroidal integer homology sphere $Y\cong M_1\cup M_2$;  each $M_i$ is homeomorphic to a trefoil exterior.  Since the trefoil is a torus knot, its complement is Seifert fibered.  Therefore, $\Sigma_2(K)$ is an irreducible toroidal graph manifold and $\pi_1(\Sigma_2(K))$ is left-orderable by  Theorem~\ref{thm:main}. We remark that this structure of $\Sigma_2(K)$ may be deduced by inspection of the branch set (see \cite{Montesinos75}, for example).

Now we can easily extend this example to obtain infinitely many $n$ such that $\Sigma_n(K)$ is left-orderable and $H_1(\Sigma_n(K);\bZ)$ finite.  A theorem of Fox states that $H_1(\Sigma_n(K);\bZ)$ will be infinite if and only if the Alexander polynomial for $K$ has zeros at $n^{\rm th}$ roots of unity \cite{Weber79}.  Since the Conway knot has Alexander polynomial 1, $H_1(\Sigma_n(K);\bZ)$ is always finite.

If $2|n$ then $\Sigma_n(K)$ is an $m$-fold cover of $\Sigma_2(K)$ where $m=\frac{n}{2}$. To see this, recall that $\pi_1(\Sigma_n(K))$ may be  viewed as the group $K^n=\ker(\pi_1(K)\to\bZ/n\bZ)/\langle\langle\mu^n\rangle\rangle$ where $\mu$ is the meridian of $K$. There is an obvious inclusion of $\ker(\pi_1(K)\to\bZ/n\bZ)\subset\ker(\pi_1(K)\to\bZ/2\bZ)$ when $2|n$, and this descends to give a homomorphism $K^n\to K^2$. In particular, there is a non-trivial  homomorphism $\pi_1(\Sigma_n(K))\to\pi_1(\Sigma_2(K))$ when $n$ is even.

Finally, since the Conway knot is prime, $\Sigma_n(K)$ is irreducible \cite{Plotnick84}.  The observation that $\pi_1(\Sigma_2(K))$ is left-orderable,  together with the homomorphism $\pi_1(\Sigma_n(K))\to\pi_1(\Sigma_2(K))$, implies that $\pi_1(\Sigma_n(K))$ is left-orderable (applying Theorem \ref{thm:loimage}) whenever $n$ is even.\end{proof}

We remark that, since the Conway knot is hyperbolic, the manifold $\Sigma_n(K)$ is hyperbolic for all sufficiently large $n$ by work of Thurston \cite{Thurston1980}.

\subsection{Hyperbolic pieces} It is natural to ask how Theorem \ref{thm:main} might generalize to toroidal 3-manifolds with hyperbolic pieces. 

\begin{theorem}\label{thm:extra}Let $Y$ be an irreducible toroidal integer homology sphere graph manifold with non-trivial JSJ decomposition, and fix a knot $K$ in some Seifert piece of $Y$. Consider $M'=Y'\smallsetminus K'$, the exterior of some knot $K'$ in an irreducible integer homology sphere $Y'$ for which $M'(\lambda')$ is prime.  Then the fundamental group of the manifold resulting from the splicing of $K$ and $K'$ is left-orderable.  \end{theorem}
\begin{proof} Let $\mu$ denote the meridian of $K$, and note that $\mu$ is a left-orderable slope by Theorem \ref{thm:main}.  Now the manifold resulting from splicing $K$ and $K'$ has left-orderable fundamental group by Corollary \ref{cor:splice}.  \end{proof}

Notice that $M'$ in Theorem \ref{thm:extra} may be hyperbolic, and that it is straightforward to iterate this construction with any number of hyperbolic pieces. This yields a large class of toroidal integer homology spheres with left-orderable fundamental groups, in the spirit of Theorem \ref{thm:main}, containing hyperbolic pieces. 

We remark that in the general there seems to be a strong interaction with coorientable taut foliations when considering the question of left-orderable fundamental groups among integer homology spheres. For example,  Roberts constructs coorientable taut foliations on certain surgeries on alternating hyperbolic knots in $S^3$ \cite[Theorem 0.1 and Theorem 0.2]{Roberts1995}, while Calegari and Dunfield  prove that an atoroidal integer homology sphere admitting a coorientable taut foliation has left-orderable fundamental group \cite[Corollary 7.6]{CD2003}. Combining these results yields a large class of hyperbolic integer homology spheres with left-orderable fundamental groups by considering $\frac{1}{n}$-surgery on any alternating hyperbolic knot in $S^3$. This provides a potentially useful tool for constructing further infinite families along the lines of Theorem \ref{thm:extra}.

\subsection{Toroidal graph manifolds in general.} While Theorem \ref{thm:JSJLO} is sufficient to deal with integer homology sphere graph manifolds, it is not sufficient to deal with all rational homology sphere graph manifolds. However, in this setting it is sometimes possible to apply Theorem 4 directly in order to left-order the fundamental group of a given graph manifold. The example that follows illustrates this fact.

Recall that the braid group $B_n$ is the group generated by $\sigma_1, \cdots \sigma_{n-1}$, subject to the relations $\sigma_i \sigma_j \sigma_i = \sigma_j \sigma_i \sigma_j$ if $|i-j|=1$, and $\sigma_i \sigma_j  = \sigma_j \sigma_i $ if $|i-j| >1$.  The braid group $B_3$ is isomorphic to the fundamental group of the complement of the trefoil; viewed as such the meridian is $\mu = \sigma_2$ and the longitude is $\lambda = \Delta^2 \sigma_2^{-6}$, where
\[ \Delta = \sigma_1 \sigma_2 \sigma_1 .
\]
Relative to this basis of the peripheral subgroup, the fibre slope of the Seifert structure is $6\mu + \lambda = \Delta^2$.

Next, we consider the twisted $I$-bundle over the Klein bottle, which has fundamental group
\[  \langle x, y | xyx^{-1} = y^{-1} \rangle,
\]
with peripheral subgroup $\langle y, x^2 \rangle$.   There are two possible Seifert structures on this manifold.  The first Seifert structure has a M\"{o}bius band as base orbifold, with no cone points, and the regular fibre in this case is represented by $y$.  The second Seifert structure has a disk as base orbifold, with two cone points of multiplicity 2.  In this case the regular fibre is given by $x^2$.

We now define our graph manifold of interest.  Let $M_1$ denote the complement of the trefoil, and $M_2$ the twisted $I$-bundle over the Klein bottle, with fundamental groups as described above.  Define
\[ M := M_1 \cup_{\phi} M_2
\]
where $\phi$ is the gluing map of their boundaries defined on the peripheral subgroups by the formula
\[ \phi( \sigma_2 ) = y^{-1}, \phi(\Delta^2) = y^{-1} x^2 .
\]
By applying the Seifert-Van Kampen theorem, we see that
\[ \pi_1(M) = \langle \sigma_1, \sigma_2, x, y | \sigma_1 \sigma_2 \sigma_1 = \sigma_2 \sigma_1 \sigma_2, xyx^{-1} = y^{-1}, \sigma_2 = y^{-1}, \Delta^2 = y^{-1}x^2 \rangle.
\]
It is not hard to check that $\pi_1(M)$ abelianizes to give $\mathbb{Z}/4\bZ$, so $M$ is not an integer homology sphere (In fact, $M$ arises from $+4$-surgery on the figure eight knot and came to the attention of the authors as a result of left-ordering manifolds arising from Dehn surgery on the figure eight knot, see \cite{BGW2010}).  Moreover, since the map $\phi$ does not send the unique Seifert fibre slope on $\partial M_1$ to either of the Seifert fibre slopes on $\partial M_2$, the manifold $M$ is not Seifert fibred, so is a toroidal graph manifold.

\begin{proposition} The fundamental group of $M$ cannot be left-ordered by applying Theorem \ref{thm:JSJLO}.
\end{proposition}
\begin{proof} In order to apply Theorem \ref{thm:JSJLO}, we must find a left-orderable slope $s$ which is mapped to a left-orderable slope $\phi(s)$ by $\phi$.  The twisted $I$-bundle over the Klein bottle has only one left-orderable slope, corresponding to the group element $y$ in the presentation above.  In particular, filling along any other slope gives a manifold with elliptic geometry and finite fundamental group, or a fundamental group that is a free product of finite groups in the case of filling along the slope $x^2$.  Therefore, if we are to apply Theorem \ref{thm:JSJLO} to the graph manifold $M$, we must take $\phi(s) = y$ to be our left-orderable slope on $\partial M_2$.  Correspondingly, we must have $s = \phi^{-1}(y) = \sigma_2 = \mu$.  However, $\mu = \sigma_2$ is not a left-orderable slope since $B_3 / \langle \langle \sigma_2 \rangle \rangle$ is trivial.   We conclude that Theorem \ref{thm:JSJLO} cannot be applied in this case.
\end{proof}

With a little work,  we can show that $\pi_1(M)$ is left-orderable by applying Theorem \ref{thm:amalgam}.  The success of this approach relies heavily on the fact that the left-orderings of $B_3$ are very well studied.  First we need some definitions and lemmas.

Recall that a word $w$ in the generators $\sigma_1, \sigma_2$ is \textit{1-positive} if $\sigma_1$ occurs with only positive exponents and \textit{1-negative} if $\sigma_1$ has only negative exponents.  We define a left-ordering of $B_3$, called the Dubrovina-Dubrovin ordering, according to the rule:
$\beta \in B_3$ is positive if $\beta$ admits a 1-positive representative word in the generators $\sigma_1, \sigma_2$, or if $\beta=\sigma_2^k$ for some $k<0$.   Denote this left-ordering of $B_3$ by $<_{DD}$.   For background on this left-ordering of $B_3$ and the related Dehornoy ordering, see \cite{Dehornoy94, DDRW08, DD01}.  An important property for our purposes is:
\begin{lemma}[Special case of Property S  \cite{DDRW08}]
\label{lem:propertyS}
Let $\beta$ be any braid,  suppose that $k>0$ and $\beta^{-1} \sigma_2^k \beta$ is not a power of $\sigma_2$.  The braid $\beta^{-1} \sigma_2^k \beta$ is 1-positive if and only if $k>0$.
\end{lemma}

\begin{lemma}[Malyutin \cite{MN2003}, see {\cite[Lemma 3.4]{DDRW08}}]  Given $\alpha, \beta \in B_3$, the following hold for every left-ordering $<$ of $B_3$:
\begin{enumerate}
\item $\alpha < \Delta^{2k}$ and $\beta < \Delta^{2l}$ implies $\alpha \beta < \Delta^{2(k+l)}$,
\item $\alpha > \Delta^{2k}$ and $\beta > \Delta^{2l}$ implies $\alpha \beta > \Delta^{2(k+l)}$,
\item $\alpha < \Delta^{2k}$ implies $\Delta^{-2k} < \alpha^{-1}$.
\end{enumerate}
\end{lemma}

We use these facts without explicit reference in the proofs that follow.

\begin{lemma}
\label{lem:conjugatebound}
Let $\beta \in B_3$ be any braid, and $k$ any integer.  Then
\[ \Delta^{-2} <_{DD} \beta^{-1} \sigma_2^k \beta <_{DD} \Delta^2.
\]
\end{lemma}
\begin{proof}
For contradiction, suppose that $\Delta^2 <_{DD} \beta^{-1} \sigma_2^k \beta$.   Since $\Delta$ is cofinal in every left-ordering of $B_3$ \cite{DDRW08}, we may choose $m$ so that $\Delta^{2m} <_{DD} \beta <_{DD} \Delta^{2m+2}$, hence $\Delta^{-2m-2} <_{DD} \beta <_{DD} \Delta^{-2m}$.

By squaring elements in the first inequality, we get $\Delta^4 <_{DD} \beta^{-1} \sigma_2^{2k} \beta$. Combining this with the lower bounds for $\beta$ and $\beta^{-1}$, we get:
\[ \Delta^2 = \Delta^{2m} \cdot \Delta^4 \cdot \Delta^{-2m-2} <_{DD} \beta \cdot \beta^{-1} \sigma_2^{2k} \beta \cdot \beta^{-1} = \sigma_2^{2k}.
\]
We conclude $\Delta^2 <_{DD} \sigma_2^{2k}$, or $\sigma_2^{-2k} \Delta^2 <_{DD} 1$, which is not possible since $\sigma_2^{-2k} \Delta^2$ is 1-positive.  This contradiction completes the proof.  The inequality $\Delta^{-2} <_{DD} \beta^{-1} \sigma_2^k \beta$ is proved in a similar fashion.
\end{proof}

\begin{lemma}
\label{lem:conjrestrict}
Every conjugate of the Dubrovina-Dubrovin ordering of $B_3$ restricts to one of the following two orderings on the subgroup $\langle \sigma_2, \Delta^2 \rangle$:
\begin{enumerate}
\item $\sigma_2^k \Delta^{2l}$ is positive if $l>0$, or $l=0$ and $k>0$, or
\item $\sigma_2^k \Delta^{2l}$ is positive if $l>0$, or $l=0$ and $k<0$.
\end{enumerate}
\end{lemma}
\begin{proof}
Let $\beta \in B_3$ be any braid.  If $\sigma_2^k \Delta^{2l}$ is positive in the conjugate ordering $<_{DD}^{\beta}$, then by definition we have
\[ \beta <_{DD} \sigma_2^k \Delta^{2l} \beta,
\]
which we can rearrange to give $\beta^{-1} \sigma_2^{-k} \beta <_{DD} \Delta^{2l}$.  By Lemma \ref{lem:conjugatebound}, either $l>0$, or $l=0$ and we have $\beta^{-1} \sigma_2^{-k} \beta <_{DD} 1$.

We now consider two cases.   First, if $\beta$ commutes with $\sigma_2$ (and hence all its powers \cite{FRZ96}), then $\beta^{-1} \sigma_2^{-k} \beta <_{DD} 1$ becomes $\sigma_2^{-k}<_{DD} 1$, which happens if and only if $k<0$.

Second, if $\beta$ does not commute with $\sigma_2$, then $\beta^{-1} \sigma_2^{-k} \beta$ is not a power of $\sigma_2$.  By definition of the DD-ordering, $\beta^{-1} \sigma_2^{-k} \beta <_{DD} 1$ implies that $\beta^{-1} \sigma_2^{-k} \beta$ is 1-negative.  By Property S, this happens if and only if $k>0$.
\end{proof}

\begin{proposition} The fundamental group
\[ \pi_1(M) = \langle \sigma_1, \sigma_2, x, y | \sigma_1 \sigma_2 \sigma_1 = \sigma_2 \sigma_1 \sigma_2, xyx^{-1} = y^{-1}, \sigma_2 = y^{-1}, \Delta^2 = y^{-1}x^2 \rangle
\]
is left-orderable.
\end{proposition}
\begin{proof}
Recall that $M$ is the union of two Seifert fibred pieces $M_1$ and $M_2$ with $\pi_1(M_1) = B_3$ and $\pi_1(M_2) = \langle x, y | xyx^{-1} = y^{-1} \rangle$, with the gluing map $\phi \langle \sigma_2 \Delta^2 \rangle \rightarrow \langle  y, x^2 \rangle$ given by the formula
\[ \phi( \sigma_2 ) = y^{-1}, \phi(\Delta^2) = y^{-1} x^2 .
\]

In order to use Theorem \ref{thm:JSJLO} and show that $\pi_1(M)$ is left-orderable, we must define normal families $L_1 \subset \LO(\pi_1(M_1))$ and $L_2 \subset\LO(\pi_1(M_2))$ that are compatible with the map $\phi$.

The first normal family, $L_1$, will be all conjugates of the Dubrovina-Dubrovin ordering of $B_3$, as defined above.  For the normal family $L_2$, consider the short exact sequence
\[ 1 \rightarrow \langle \langle y \rangle \rangle \rightarrow \pi_1(M_2) \stackrel{q}{\rightarrow} \langle x \rangle \rightarrow 1.
\]
We can use this short exact sequence to define two left-orderings $<_1$ and $<_2$ of $\pi_1(M_2)$ that are conjugate to each other by the action of $x$.  Given $g \in \pi_1(M_2)$, declare $1 <_1 g$ if $q(g) = x^l$, $l>0$, or if $q(g) = 1$ and $g=y^k$ with $k>0$.  Similarly, define $1 <_2 g$ if $q(g) = x^l$, $l>0$, or if $q(g) = 1$ and $g=y^k$ with $k<0$, so that the sign of $y$ is opposite in $<_2$.  It is not hard to check that  $L_2 = \{ <_1, <_2 \}$ is a normal family of $ \LO(\pi_1(M_2))$.

Now we check that the normal families $L_1$ and $L_2$ are compatible with the map $\phi$.  To this end, choose a left-ordering $<$ from the family $L_1$.  There are two cases to consider.

\noindent \textbf{Case 1.}  The ordering $<$ is conjugate to $<_{DD}$ by a braid $\beta$ that  does not commute with $\sigma_2$.    In this case, we use the ordering $<_2$ in $L_2$ to demonstrate compatibility.  Choose $h = \sigma_2^k \Delta^{2l} \in \langle \sigma_1 \Delta^2 \rangle$ satisfying $1<h$, then by Lemma \ref{lem:conjrestrict} either $l>0$ or $l=0$ and $k>0$.  Then
\[ \phi(h) = y^{-k} (y^{-1}x^2)^l,
\]
and we wish to check that $1 <_2 y^{-k} (y^{-1}x^2)^l$.  Observe that $q(y^{-k} (y^{-1}x^2)^l) = x^{2l}$, so $1 <_2 y^{-k} (y^{-1}x^2)^l$ if $l>0$.  If $l=0$, then  $ y^{-k} (y^{-1}x^2)^l = y^{-k}$, and $1 <_2 y^{-k}$ because $k>0$.

\noindent \textbf{Case 2.}   The ordering $<$ is conjugate to $<_{DD}$ by a braid $\beta$ that  commutes with $\sigma_2$.  In this case we use $<_1$ in $L_1$ to demonstrate compatibility with $\phi$.  The calculation is nearly identical to Case 1 (with the signs of $k$ reversed), and is left to the reader.
\end{proof}

\bibliographystyle{plain}
\bibliography{JSJLO}

\begin{thebibliography}{10}

\bibitem{BG08}
V.~V. Bludov and A.~M.~W. Glass.
\newblock Conjugacy in lattice-ordered groups and right ordered groups.
\newblock {\em J. Group Theory}, 11(5):623--633, 2008.

\bibitem{BG2009}
V.~V. Bludov and A.~M.~W. Glass.
\newblock Word problems, embeddings, and free products of right-ordered groups
  with amalgamated subgroup.
\newblock {\em Proceedings of the London Mathematical Society}, 99(3):585--608,
  2009.

\bibitem{BGW2010}
Steven Boyer, Cameron Gordon, and Liam Watson.
\newblock {On L-spaces and left-orderable fundamental groups}.
\newblock {Preprint}.

\bibitem{BRW2005}
Steven Boyer, Dale Rolfsen, and Bert Wiest.
\newblock Orderable 3-manifold groups.
\newblock {\em Ann. Inst. Fourier (Grenoble)}, 55(1):243--288, 2005.

\bibitem{CD2003}
Danny Calegari and Nathan~M. Dunfield.
\newblock Laminations and groups of homeomorphisms of the circle.
\newblock {\em Invent. Math.}, 152(1):149--204, 2003.

\bibitem{Chiswell2010}
I.~M. Chiswell.
\newblock Right orderability and graphs of groups.
\newblock {\em Journal of Group Theory}, pages ----, 2011/05/12 2010.

\bibitem{DPT2005}
Mieczys{\l}aw~K. D{\c{a}}bkowski, J{\'o}zef~H. Przytycki, and Amir~A. Togha.
\newblock Non-left-orderable 3-manifold groups.
\newblock {\em Canad. Math. Bull.}, 48(1):32--40, 2005.

\bibitem{Dehornoy94}
Patrick Dehornoy.
\newblock Braid groups and left distributive operations.
\newblock {\em Trans. Amer. Math. Soc.}, 345(1):115--150, 1994.

\bibitem{DDRW08}
Patrick Dehornoy, Ivan Dynnikov, Dale Rolfsen, and Bert Wiest.
\newblock {\em Ordering braids}, volume 148 of {\em Mathematical Surveys and
  Monographs}.
\newblock American Mathematical Society, Providence, RI, 2008.

\bibitem{DD01}
T.~V. Dubrovina and N.~I. Dubrovin.
\newblock On braid groups.
\newblock {\em Mat. Sb.}, 192(5):53--64, 2001.

\bibitem{Eftekhary2009}
Eaman Eftekhary.
\newblock Seifert fibered homology spheres with trivial {H}eegaard {F}loer
  homology groups.
\newblock Preprint, arXiv:0909.3975.

\bibitem{FRZ96}
Roger Fenn, Dale Rolfsen, and Jun Zhu.
\newblock Centralisers in the braid group and singular braid monoid.
\newblock {\em Enseign. Math. (2)}, 42(1-2):75--96, 1996.

\bibitem{Heil1974}
Wolfgang Heil.
\newblock Elementary surgery on {S}eifert fiber spaces.
\newblock {\em Yokohama Math. J.}, 22:135--139, 1974.

\bibitem{JS1976}
William Jaco and Peter~B. Shalen.
\newblock Seifert fibered spaces in irreducible, sufficiently-large
  {$3$}-manifolds.
\newblock {\em Bull. Amer. Math. Soc.}, 82(5):765--767, 1976.

\bibitem{Johannson1975}
Klaus Johannson.
\newblock \'{E}quivalences d'homotopie des vari\'et\'es de dimension {$3$}.
\newblock {\em C. R. Acad. Sci. Paris S\'er. A-B}, 281(23):Ai, A1009--A1010,
  1975.

\bibitem{MN2003}
A.~V. Malyutin and N.~Yu. Netsvetaev.
\newblock Dehornoy order in the braid group and transformations of closed
  braids.
\newblock {\em Algebra i Analiz}, 15(3):170--187, 2003.

\bibitem{Kourovka02}
V.~D. Mazurov and E.~I. Khukhro, editors.
\newblock {\em Kourovskaya tetrad}.
\newblock Rossi\u\i skaya Akademiya Nauk Sibirskoe Otdelenie, Institut
  Matematiki im. S. L. Soboleva, Novosibirsk, augmented edition, 2002.
\newblock Nereshennye voprosy teorii grupp. [Unsolved problems in group
  theory].

\bibitem{Montesinos75}
Jos{\'e}~M. Montesinos.
\newblock Surgery on links and double branched covers of {$S^{3}$}.
\newblock In {\em Knots, groups, and {$3$}-manifolds ({P}apers dedicated to the
  memory of {R}. {H}. {F}ox)}, pages 227--259. Ann. of Math. Studies, No. 84.
  Princeton Univ. Press, Princeton, N.J., 1975.

\bibitem{OSz-rational}
Peter Ozsv{\'a}th and Zolt{\'a}n Szab{\'o}.
\newblock {Knot Floer homology and rational surgeries}.
\newblock {Preprint, arXiv:0504404}.

\bibitem{OSz2004-properties}
Peter Ozsv{\'a}th and Zolt{\'a}n Szab{\'o}.
\newblock Holomorphic disks and three-manifold invariants: properties and
  applications.
\newblock {\em Ann. of Math. (2)}, 159(3):1159--1245, 2004.

\bibitem{Peters2009}
Thomas Peters.
\newblock {On L-spaces and non left-orderable 3-manifold groups}.
\newblock {Preprint, arXiv:0903.4495}.

\bibitem{Plotnick84}
Steven~P. Plotnick.
\newblock Finite group actions and nonseparating {$2$}-spheres.
\newblock {\em Proc. Amer. Math. Soc.}, 90(3):430--432, 1984.

\bibitem{Roberts1995}
Rachel Roberts.
\newblock Constructing taut foliations.
\newblock {\em Comment. Math. Helv.}, 70(4):516--545, 1995.

\bibitem{Scott1973}
G.~P. Scott.
\newblock Compact submanifolds of {$3$}-manifolds.
\newblock {\em J. London Math. Soc. (2)}, 7:246--250, 1973.

\bibitem{Thurston1980}
William~P. Thurston.
\newblock The geometry and topology of three-manifolds, 1980.
\newblock Lecture notes.

\bibitem{Vino49}
A.~A. Vinogradov.
\newblock On the free product of ordered groups.
\newblock {\em Mat. Sbornik N.S.}, 25(67):163--168, 1949.

\bibitem{Waldhausen1967}
Friedhelm Waldhausen.
\newblock Eine {K}lasse von {$3$}-dimensionalen {M}annigfaltigkeiten. {I},
  {II}.
\newblock {\em Invent. Math. 3 (1967), 308--333; ibid.}, 4:87--117, 1967.

\bibitem{Watson2009}
Liam Watson.
\newblock {\em {Involutions on 3-manifolds and Khovanov homology}}.
\newblock {PhD}, {Universit\'e du Qu\'ebec \`a Montr\'eal}, 2009.

\bibitem{Weber79}
Claude Weber.
\newblock Sur une formule de {R}. {H}. {F}ox concernant l'homologie des
  rev\^etements cycliques.
\newblock {\em Enseign. Math. (2)}, 25(3-4):261--272 (1980), 1979.

\bibitem{ZWu2009}
Zhongtao Wu.
\newblock Cosmetic surgery in integral homology {L}-spaces.
\newblock {Preprint, arXiv:0911.5333v1}.

\end{thebibliography}
\end{document}